\documentclass{amsart}

\usepackage[all]{xy}
\usepackage{graphicx}
\usepackage{amssymb}
\usepackage{mathrsfs}
\usepackage{multirow}
\usepackage{float}
\usepackage{tikz-cd}
\usepackage{adjustbox}
\usepackage{amsthm}
\usepackage{accents}
\usepackage{mathtools}
\usepackage{enumitem,cite,array}
\usepackage[new]{old-arrows}
\usepackage{tikz}
\usetikzlibrary{matrix,arrows}
\usepackage{amsmath}

\usepackage[numbers,sort&compress]{natbib}
\usepackage[bookmarksnumbered, bookmarksopen,
colorlinks,citecolor=blue,linkcolor=blue,backref]{hyperref}

\newcounter{RomanNumber}
\newcommand{\MyRoman}[1]{\setcounter{RomanNumber}{#1}\Roman{RomanNumber}}

\newtheorem{theorem}{Theorem}
\newtheorem{lemma}[theorem]{Lemma}

\theoremstyle{definition}
\newtheorem{definition}[theorem]{Definition}

\newtheorem{proposition}[theorem]{Proposition}
\newtheorem{corollary}[theorem]{Corollary}

\newtheorem{question}[theorem]{Question}

\newtheorem{conjecture}[theorem]{Conjecture}
\theoremstyle{remark}
\newtheorem{remark}[theorem]{Remark}

\theoremstyle{notation}

\newcommand{\be}{\begin{equation}}
\newcommand{\ee}{\end{equation}}

\newcommand{\CC}{\mathbb{C}}

\newcommand{\ZZ}{\mathbf{Z}}
\newcommand{\h}{\begin{eqnarray*}}
\newcommand{\e}{\end{eqnarray*}}
\newcommand{\ind}{\mathrm{Ind}}
\newcommand{\sig}{\mathrm{Sig}}

\usepackage[utf8]{inputenc}
\usepackage{chngcntr}
\usepackage{apptools}
\AtAppendix{\counterwithin{theorem}{section}}
\usepackage[toc,page]{appendix}

\newcommand{\qqed}{\hfill\Box}

\begin{document}

\keywords{}

\title
{On characteristic numbers of $24$ dimensional String manifolds}

\author{Fei Han}
\address{Department of Mathematics,
National University of Singapore, Singapore 119076}
\curraddr{}
\email{mathanf@nus.edu.sg}
\thanks{}
\urladdr{https://blog.nus.edu.sg/mathanf/}


 \author{Ruizhi Huang}
\address{Ruizhi Huang, Institute of Mathematics and Systems Sciences, Chinese Academy of Sciences, Beijing 100190, China}
\email{huangrz@amss.ac.cn}  
\urladdr{https://sites.google.com/site/hrzsea}
\thanks{}

\date{}

\maketitle

\begin{abstract}
In this paper, we study the Pontryagin numbers of $24$ dimensional String manifolds. In particular, we find representatives of an integral basis of the String cobrodism group at dimension $24$, based on the work of Mahowald-Hopkins \cite{MH02}, Borel-Hirzebruch \cite{BH58} and Wall \cite{Wall62}. This has immediate applications on the divisibility of various characteristic numbers of the manifolds. In particular, we establish the $2$-primary divisibilities of the signature and of the modified signature coupling with the integral Wu class of Hopkins-Singer \cite{HS05}, and also the $3$-primary divisibility of the twisted signature. Our results provide potential clues to understand a question of Teichner.
\end{abstract}

\tableofcontents

\section{Introduction}
Let $M$ be a $4m$ dimensional oriented closed manifold. $M$ is called \textit{Spin} if its second Stiefel-Whitney class vanishes: $\omega_2(M)=0$. To investigate the geometry and topology of $M$, it is classical to study its characteristic numbers as cobordism invariants. Among others, there are two important types of characteristic numbers, namely {\it the twisted $A$-hat genus $\widehat{A}(M, E)$} and {\it the twisted signature ${\rm Sig}(M, E)$} for any given complex bundle $E$ over $M$ (see Appendix \hyperref[backsec]{A} for explicit definitions). For instance, there are the twisted genera coupling with bundles naturally constructed from the tangent bundle $TM$ of $M$
\[
\begin{split}
\widehat{A}(M, T^{i}\otimes \wedge^{j}\otimes S^{k}) &:=\widehat{A}(M, \otimes^{i}T_{\mathbb{C}}M\otimes \wedge^j(T_{\mathbb{C}}M)\otimes S^k(T_{\mathbb{C}}M)),\\
{\rm Sig}(M, T^{i}\otimes \wedge^{j}\otimes S^{k}) &:={\rm Sig}(M, \otimes^{i}T_{\mathbb{C}}M\otimes \wedge^j(T_{\mathbb{C}}M)\otimes S^k(T_{\mathbb{C}}M)),
\end{split}
\]
where $\wedge^j(T_{\mathbb{C}}M)$ and $S^k(T_{\mathbb{C}}M)$ are the $j$-th exterior and $k$-th symmetric powers of $T_{\mathbb{C}}M$ respectively.

As a twisted $\widehat{A}$-genus, the famous Witten genus $W(M)$ (\cite{W}; see Appendix \hyperref[backsec]{A} for explicit definition) possesses nice properties especially when $M$ is \textit{String}, that is, half of the first Pontryagin class vanishes: $\frac{p_1(M)}{2}=0$. For instance, in the String case, the Witten genus $W(M)$ is a
modular form of weight $2m$ over $SL(2,\mathbb{Z})$ with integral
Fourier expansion (\cite{Za}). The homotopy theoretical
refinement of the Witten genus on String manifolds leads to the
theory of tmf ({\em topological modular form}) developed by Hopkins and
Miller \cite {Hop}. The String condition is the orientablity condition for this generalized cohomology theory.

String manifolds of dimension $24$ are of special interest. For instance, in this dimension, one has (cf. page $85$-$87$ in \cite{HBJ94})
\[
W(M)=\widehat{A}(M)\bar{\Delta}+\widehat{A}(M,T)\Delta,
\] where $\bar{\Delta}=E_4^3-744\cdot \Delta$
with $E_4$ being the Eisenstein series of weight $4$ and $\Delta$ being the modular discriminant of weight $12$ (see Section \ref{introsubsecTZ} for definitions).
Hirzebruch raised his prize question in \cite{HBJ94} that whether there exists a $24$ dimensional compact String manifold $M$ such that $W(M)=\bar{\Delta}$ (or equivalently $\widehat{A}(M)=1, \widehat{A}(M, T)=0$) and the Monster group acts on $M$ as self-diffeomorphisms.
The existence of such manifold was confirmed by Mahowald-Hopkins \cite{MH02}. Indeed, they determined the image of Witten genus at this dimension via $tmf$. However, the part of the question concerning the Monster group is still open.

In this paper, we study the Pontryagin numbers of $24$ dimensional String manifolds from the perspective of algebraic topology. 
Combining the works of Mahowald-Hopkins \cite{MH02}, Borel-Hirzebruch \cite{BH58} and Wall \cite{Wall62} we find representatives of an integral basis of the String cobordism
group at dimension 24.
This has immediate applications to the divisibility of various characteristic numbers of the $24$ dimensional String manifolds. It also provides potential clue for understanding a question of Teichner (see Subsection \ref{introsubsecTZ}).

\subsection{Basis of String cobordism at dimension $24$}
Let $\Omega_{24}^{String}$ be the String cobordism group of dimension $24$. 
By the calculation of Gorbounov-Mahowald \cite{GM95}, it is known that as a group
\[
\Omega_{24}^{String}\cong \mathbb{Z}\oplus \mathbb{Z}\oplus\mathbb{Z}\oplus\mathbb{Z}.
\]
In \cite{MH02} Mohowald-Hopkins determined two out of the four generators as the
coimage of the Witten genus. In particular, by homotopy arguments they constructed two $24$ dimensional String manifolds with explicit Pontryagin numbers, which we denote by $M_1$ and $M_2$ respectively in Section \ref{M12sec}. It should be emphasized that the geometry of $M_1$ and $M_2$ is still mystery, which is crucial to the prize question of Hirzebruch.
In section \ref{M3sec} and \ref{M4sec}, we construct the remaining two generators $M_{3}$ and $M_{4}$, and compute their Pontryagin numbers, respectively.

Our main result is that these $4$ manifolds, $M_1$, $M_2$, $M_3$ and $M_4$, represent an integral basis of $\Omega_{24}^{String}$. Indeed, this integral basis realizes a particular basis of all possible integral Pontryagin numbers of String $24$-manifolds, consisting of $\widehat{A}(-)$, $\frac{1}{24}\widehat{A}(-,T)$, $\widehat{A}(-, \wedge^2)$, and $\frac{1}{8}{\rm Sig}(-)$. Here, for any $M\in \Omega_{24}^{String}$, $\frac{1}{24}\widehat{A}(M,T)\in \mathbb{Z}$ was proved by Mahowald-Hopkins (\cite{MH02}; also see the discussion in Subsection \ref{introsubsecTZ}), and $\frac{1}{8}{\rm Sig}(-)\in \mathbb{Z}$ is showed in Lemma \ref{sig8lemma} (cf. Section $7$ of \cite{MH02}).
In particular, we completely understand the Pontryagin numbers of String manifolds at dimension $24$.

\begin{theorem}\label{basisthm}
The correspondence 
$\kappa :\Omega
_{24}^{String}\rightarrow $ $\mathbb{Z}\oplus \mathbb{Z}\oplus \mathbb{Z}%
\oplus \mathbb{Z}$ by
\[
\kappa (M)=(\widehat{A}(M),\frac{1}{24}\widehat{A}(M,T),%
\widehat{A}(M,\Lambda ^{2}),\frac{1}{8}{\rm Sig}(M))
\]
is an isomorphism of abelian groups. Moreover, there exist two explicitly constructed manifolds $M_{3},M_{4}\in \ker W$ such that
\[
K:=
\begin{pmatrix}
\kappa (M_{1})\\
\kappa (M_{2})\\
\kappa (M_{3})\\
\kappa (M_{4})
\end{pmatrix}^\tau
=
\begin{pmatrix}
0 & 1  & 0 & 0\\
-1 & 0 &0 & 0\\
2^3\cdot 3^3\cdot 5  & 2^2\cdot 3\cdot 17 \cdot 1069 & -1 & 0\\
2^8\cdot 3\cdot 61   & 2^8\cdot 5\cdot 37 & 2^2\cdot 7 & 1
\end{pmatrix}.
\]
\end{theorem}
Two notable consequences of Theorem 1 are as follows.
\begin{corollary}\label{basiscoro1}
For any rational homogeneous polynomial $P$ in the
Pontryagin classes with degree $24$ and with $P(M)\in \mathbb{Z}$ for all $%
M\in \Omega _{24}^{String}$, there exist unique integers $a_{1},\cdots
,a_{4} $ such that
\[
P(M)=a_1\widehat{A}(M)+\frac{1}{24}a_2\widehat{A}(M,T)+a_3\widehat{A}(M, \wedge^2)+\frac{1}{8}a_4{\rm Sig}(M).
\] 
\end{corollary}
\begin{corollary}\label{basiscoro2}
The four manifolds $M_{i}$ in Theorem \ref{basisthm}
form a basis of the group $\Omega _{24}^{String}$.
\end{corollary}

\subsection{$2$-primary divisibility of signature}
Theorem \ref{basisthm} has strong implications on the characteristic numbers of $24$-dimensional String manifolds.
The first example concerns the signatures of the manifolds. Indeed, for $24$ dimensional String manifold $M$, it is always true that (Lemma \ref{sig8lemma}, or Section $7$ of \cite{MH02} by Mahowald-Hopkins)
\[
8~|~{\rm Sig}(M),
\] 
which is optimal because ${\rm Sig}(M_4)=8$ as indicated in Theorem \ref{basisthm}. Nevertheless, if we have more information about the topology of the manifold, say the divisibility of its second Pontryagin class, it can be expected that there is higher divisibility of the signature. To be precise, let $m$ and $n$ be two positive integer. By abuse of notation, we can define a new integer $(\frac{n}{m})$ by
\[
\nu_{p}(\frac{n}{m})=\left\{\begin{array}{ll}
\nu_{p}(n)-\nu_{p}(m) & ~{\rm if}~\nu_{p}(n)\geq \nu_{p}(m),\\
0 & ~{\rm if}~\nu_{p}(n)< \nu_{p}(m),
\end{array}
\right.
\]
where $\nu_p(k)$ denotes the exponent of the largest power of the prime $p$ that divides $k$.
\begin{theorem}\label{sigthm}
Let $M$ be a $24$ dimensional String manifold. If its $8$-th Stiefel-Whitney class vanishes: $\omega_8(M)=0$, then
\[
32~|~{\rm Sig}(M).
\]
Furthermore, if the second Pontryagin class $p_{2}(M)$ is divisible by a positive integer $n$, then 
\[
\big(\frac{n^3}{2^2\cdot 3^5\cdot 5^3\cdot 41}\big)~|~{\rm Sig}(M).
\]
\end{theorem}
Let us remark that it is known that for a String manifold $M$ of any dimension, $6~|~p_2(M)$ by Borel-Hirzebruch \cite{BH58} (also see Li-Duan \cite{LD91}).

\subsection{$2$-primary divisibility of modified signature}
Let $BSpin$ be the classifying space of the spinor groups in the stable
range. In \cite{HS05} Hopkins-Singer constructed a universal integral lift $\nu_{4k}^{Spin}\in H^{4k}(BSpin;\mathbb{Z})$ of the mod-$2$ Wu class with degree $4k$ for each positive integer $k$, and the total integral Spin Wu class
\[
\nu_t^{{\rm Spin}}=1+\nu_4^{{\rm Spin}}+\nu_8^{{\rm Spin}}+\nu_{12}^{{\rm Spin}}+\cdots
\]
has the characteristic series of the form
\[
g(x)=1+\frac{1}{2}x^2+\frac{11}{8}x^4+\frac{37}{16}x^6+\frac{691}{128}x^8+\frac{2847}{256}x^{10}+\cdots.
\]
In term of these classes we define for a Spin $8k$-manifold $M$ the {\it modified signature} by the formulae
\[
{\rm Sig}(M, \nu):={\rm Sig}(M)-\langle \nu_{4k}^{{\rm Spin}}(M)\cup \nu_{4k}^{{\rm Spin}}(M), [M]\rangle, ~\  \nu_{4k}^{{\rm Spin}}(M):=f^{\ast}(\nu_{4k}^{{\rm Spin}}),
\] 
where $f: M\rightarrow BSpin$ is the classifying map of the normal
bundle of $M$ in the stable range.
It is a classical result that ${\rm Sig}(M, \nu)$ is divisible by $8$ for Spin manifolds. For our String manifold $M$ however, we actually can get higher divisibility. 
\begin{theorem}\label{modsigthm}
Let $M$ be a $24$ dimensional String manifold. Then
\[
32~|~{\rm Sig}(M, \nu).
\]
\end{theorem}
The divisibility in Theorem \ref{modsigthm} is optimal, as we shall see in its proof in Section \ref{proofthmsec} that the manifold $67 M_3+3M_4\in \Omega^{String}_{24}$ has the modified
signature exactly equal to 32.
\subsection{$3$-primary divisibility of twisted signature}
In \cite{CH15}, Chen-Han studied the mod-$3$ congruence properties of certain twisted signature of $24$ dimensional String manifolds. By the techniques of modular forms in index theory, they showed that
\[
3~|~{\rm Sig}(M, \wedge^2)
\]
for any $24$ dimensional String manifold $M$, and this is the best possible. By Theorem \ref{basisthm}, it is easy to give a topological proof of this result by straightforward computation. Indeed, it can be showed that (see Remark \ref{sigwedge2remark})
\[
96~|~{\rm Sig}(M, \wedge^2)
\]
in general, while for the generator $M_{1}$
\[
3^2~\nmid~{\rm Sig}(M_1, \wedge^2).
\] 
However, as in Theorem \ref{sigthm}, with $3$-primary divisibility of the second Pontryagin class $p_2(M)$, we can obtain higher divisibility for the twisted signature ${\rm Sig}(M, \wedge^2)$.
\begin{theorem}\label{twistedsigthm}
Let $M$ be a $24$ dimensional String manifold. If $3^{k+1}~|~p_2(M)$ ($k\geq 1$), then
\[
3^{3k-1}~|~{\rm Sig}(M, \wedge^2).
\]
\end{theorem}

\subsection{Discussions on a question of Teichner}\label{introsubsecTZ}
The original proof of Mahowald-Hopkins \cite{MH02} on the fact (observed by Teichner \cite{Teichner}) that
\begin{equation}\label{mh24modeq}
24~|~\widehat{A}(M, T)
\end{equation}
for any $24$ dimensional String manifold is of homotopy theoretical argument. It is based on the homotopy theory of Witten genus via $tmf$. Actually let $\Omega^{String}_{4k}$ be the string cobordism group in dimension $4k$. Let $MF_{2k}^{\ZZ}(SL(2, \ZZ))$ be the space of modular forms of weight $2k$ over $SL(2, \ZZ)$ with integral Fourier expansion. The Witten genus $W$ is the composition of the maps (\cite{Ho}):
$$\xymatrix@C=0.5cm{\Omega^{String}_{4k}\ar[r]^{\sigma\ \ }& tmf^{-4k}(pt)\ar[r]^{e\ \ \ \ }& MF_{2k}^{\ZZ}(SL(2, \ZZ))},$$ where $\sigma$ is the {\it refined Witten genus} and $e$ is the edge homomorphism in a spectral sequence. Hopkins and Mahowald (\cite{Ho}) show that $\sigma$ is surjective. For $i,l\geq 0, j=0,1$, define $$ a_{i,j,l}= \left\{\begin{array}{ccc}
                                      1& & \ \ \ \ \  \ i>0, j=0\\
                                      2& & j=1\ \ \\
                                      24/\mathrm{gcd}(24,l)& & i,j=0
                                     \end{array}\right..
                                     $$
Hopkins and Mahowald also show that the image of $e$ (and therefore the image of the Witten genus) has a basis given by monomials
\be a_{i,j,l}E_4(\tau)^iE_6(\tau)^j\Delta(\tau)^l,\ \ \ \ i,l\geq 0, j=0,1, \ee
where
\h
\begin{split}
&E_{4}(\tau)=1+240(q+9q^{2}+28q^{3}+\cdots),\\
&E_{6}(\tau)=1-504(q+33q^{2}+244q^{3}+\cdots)\\
\end{split}
\e
are the Eisenstein series and $\Delta(\tau)=q\prod_{n\geq 0}(1-q^n)^{24}$ is the modular discriminant. Their weights are 4,6, 12 respectively. In dimension 24, the image of the Witten genus is spanned by the monomials $E_4(\tau)^3$, $24\Delta(\tau)$ and since $\widehat{A}(M, T)-24\widehat{A}(M)$ is the coefficient of $q$ in the expansion of the Witten genus, $\widehat{A}(M, T)$ is divisible by 24. This observation was due to Teichner \cite{Teichner}, who consequently raised the following question,
\begin{question}\label{Techquestion}
Can we give a geometric proof of (\ref{mh24modeq})?
\end{question} 
Zhang \cite{Z0} suggested that we may look at the geometry of this divisibility from the index theoretical point of view, that is, to study if we can express $\frac{1}{24}\widehat{A}(M,T)$ as an integral linear combination of indices of twisted Dirac operators or twisted signature operators. 

Indeed, we are able to show with the help of computer program that for the generator $M_1$ of Hopkins-Mahowald, when $i+j+k\leq 5$, one has
\begin{equation}\label{lessthan5}
\begin{split}
&24~|~\widehat{A}(M_1, T^{i}\otimes \wedge^{j}\otimes S^{k}),\\
&24~|~{\rm Sig}(M_1, T^{i}\otimes \wedge^{j}\otimes S^{k}).
\end{split}
\end{equation}
This motivates us to conjecture that
\begin{conjecture}\label{achtm1conjecture} For  any non-negative integer $i$, $j$ and $k$, 
\begin{equation}\label{conj24}
\begin{split}
&24~|~\widehat{A}(M_1, T^{i}\otimes \wedge^{j}\otimes S^{k}),\\
&24~|~{\rm Sig}(M_1, T^{i}\otimes \wedge^{j}\otimes S^{k}).
\end{split}
\end{equation}
\end{conjecture}

If the conjecture is true, then $\frac{1}{24}\widehat{A}(M_1,T)$ can not be written as a linear combination of $\widehat{A}(M_1, T^{i}\otimes \wedge^{j}\otimes S^{k})$ or ${\rm Sig}(M_1, T^{i}\otimes \wedge^{j}\otimes S^{k})$ with integral coefficients. Otherwise, suppose we have an index formula for $\frac{1}{24}\widehat{A}(M,T)$ of this form, then it follows from 
(\ref{conj24}) that $\frac{1}{24}\widehat{A}(M_1,T)$ must be divisible by $24$. However, it is equal to $-1$ by Theorem \ref{basisthm}, hence a contradiction. Indeed, from the discussion, if the conjecture is true, then for any $k\geq 2$,
$\frac{1}{k}\widehat{A}(M_1,T)$ can not written as a linear combination of $\widehat{A}(M_1, T^{i}\otimes \wedge^{j}\otimes S^{k})$ or ${\rm Sig}(M_1, T^{i}\otimes \wedge^{j}\otimes S^{k})$ with integral coefficients. 

This suggests that if we want to express $\frac{1}{24}\widehat{A}(M_1,T)$ as linear combination of indices of twisted Dirac operators or twisted signatures, one need to look at more types of twistings in addition to the bundles of the form $T^{i}\otimes \wedge^{j}\otimes S^{k}$.

\subsection{Organization of the paper}The paper is organized as follows. 
In Section \ref{proofthmsec}, we prove the $4$ theorems in the introduction section by the knowledge presented in Section \ref{M12sec}, \ref{M3sec} and \ref{M4sec}.
In Section \ref{M12sec} we summarize part of the work of Mahowald-Hopkins \cite{MH02} on the coimage of Witten genus at dimension $24$. In particular, we review the Pontryagin numbers of $M_1$ and $M_2$. In Section \ref{M3sec} and Section \ref{M4sec} we construct $M_3$ and $M_4$, and compute their Pontryagin numbers respectively. 
For $M_4$ more explicitly, we apply Wall's classification on $(n-1)$-connected $2n$-manifolds \cite{Wall62} to construct $M_4$ as $F_4$-$\mathbb{O}P^2$-bundle (\ref{f4op2bundledia}), and apply the classical Borel-Hirzebruch algorithm \cite{BH58} to calculate its Pontryagin classes.
This is divided into $4$ steps in Section \ref{M4sec}. We end the paper with Appendix \hyperref[backsec]{A} explaining the geometric and analytic aspects of twisted $A$-hats and twisted signatures together with their definitions.

\bigskip

\noindent{\bf Acknowledgements.}
Fei Han was partially supported by the grant AcRF R-146-000-263-114 from National University of Singapore. He thanks Dr. Qingtao Chen, Prof. Huitao Feng, Prof. Kefeng Liu and Prof. Weiping Zhang for helpful discussions. 

Ruizhi Huang was supported by Postdoctoral International Exchange Program for Incoming Postdoctoral Students under Chinese Postdoctoral Council and Chinese Postdoctoral Science Foundation.
He was also supported in part by Chinese Postdoctoral Science Foundation (Grant nos. 2018M631605 and 2019T120145), and National Natural Science Foundation of China (Grant nos. 11801544 and 11688101), and ``Chen Jingrun'' Future Star Program of AMSS. He would like to thank Prof. Haibao Duan for discussions on topology of Lie groups, and to Prof. Yang Su for several points on geometric topology of manifolds.

Both authors would like to thank the Mathematical Science Research Center at Chongqing Institute of Technology for hospitality during their visit. They are also thankful to Prof. Zhi Lv for inspiring discussion on cobordisms.


\numberwithin{equation}{section}
\numberwithin{theorem}{section}
\section{Proof of Theorem \ref{basisthm}, \ref{sigthm}, \ref{modsigthm} and \ref{twistedsigthm}}
\label{proofthmsec}
Before proving the main theorems stated in the Introduction, we
summarize necessary results whose proofs are postponed to Sections \ref{M12sec}-\ref{M4sec}.

\begin{theorem}\label{Minumberthm}
There exist four elements $M_{i}\in \Omega
_{24}^{String}$, $1\leq i\leq 4$, whose Pontryagin numbers and
Witten genus are given by the table below:

\[%
\begin{array}{c|cccc}
M & M_{1} & M_{2} & M_{3} & M_{4} \\ 
\hline
p_{2}^{3} & 2^{13}\cdot 3^{5}\cdot 5^{3} & -2^{13}\cdot 3^{5}\cdot
5^{3}\cdot 41 & 2^{7}\cdot 3^{5}\cdot 5 & 3888 \\ 
p_{3}^{2} & 2^{10}\cdot 3^{4}\cdot 5^{2}\cdot 7^{2} & 2^{10}\cdot 3^{4}\cdot
5^{2}\cdot 7^{2}\cdot 31 & 0 & 200 \\ 
p_{2}p_{4} & 2^{12}\cdot 3^{5}\cdot 5^{3} & -2^{12}\cdot 3^{5}\cdot
5^{3}\cdot 41 & 2^{5}\cdot 3^{3}\cdot 5^{3} & 2868 \\ 
p_{6} & 2^{9}\cdot 3^{4}\cdot 5^{2}\cdot 89 & -2^{9}\cdot 3^{4}\cdot
5^{2}\cdot 11^{2} & 2^{5}\cdot 3^{3}\cdot 5\cdot 13 & 1958 \\ 
W & -24\Delta  & \overline{\Delta } & 0 & 0 
\end{array},%
\]
where 
$
\bar{\Delta}=E_4^3-744\cdot \Delta
$
with $E_4$ the Eisenstein series of weight $4$ and $\Delta$ the cusp form of weight $12$.
\end{theorem}
\begin{proof}
The manifolds $M_1$ and $M_2$ are constructed and studied by Mahowald-Hopkins \cite{MH02} from the homotopy theoretical point of view, and their Pontryagin numbers and Witten genera are summarized in Theorem \ref{MHthm} and Lemma \ref{ponnom1m2lemma}. In particular, $M_1$ and $M_2$ form a basis of the coimage of the Witten genus at dimension $24$. On the other hand, $M_3$ and $M_4$ form a basis of the kernel of the Witten genus which are explicitly constructed in Section \ref{M3sec} and Section \ref{M4sec} respectively. 
Their Pontryagin numbers are summarized in Lemma \ref{ponnom3lemma} and Lemma \ref{ponnom4lemma}.
\end{proof}

\begin{lemma}\label{sig8lemma}
Let $M$ be a $24$ dimensional String manifold. Then
\[
8~|~{\rm Sig}(M).
\]
\end{lemma}
The lemma was implicitly proved by Mahowald-Hopkins \cite{MH02} without statement. Here we give an alternative proof.
\begin{proof}
Recall that any integral lift of the middle Wu class $\nu_{12}(M)\in H^{12}(M;\mathbb{Z}/2)$ is a characteristic element for the intersection form $I(M)$ of $M$ over $\mathbb{Z}$. However, since $M$ is String (by Bensen-Wood \cite{BW95} or Duan \cite{Duan18})
\[
0=q_1(M)\equiv \omega_4(M)~{\rm mod}~2,
\]
which implies that 
\[
\omega_6(M)=Sq^2\omega_4(M)=0
\]
as well. It follows that
\[
\nu_{12}(M)=\omega_6^2(M)=0.
\]
In particular, the trivial cohomology class $0$ acts as an integral lift of $\nu_{12}(M)$ and hence is a characteristic. It follows that the intersection form $I(M)$ is of even type, and 
\[
{\rm Sig}(M)\equiv I(M)(0, 0)=0 ~{\rm mod}~8.
\]
\end{proof}

\noindent{\it Proof of Theorem \ref{basisthm}, Corollary \ref{basiscoro1} and Corollary \ref{basiscoro2}.}\quad 
Recall we have $4$ particular String manifolds $M_1$, $M_2$, $M_3$ and $M_4$ of dimension $24$, the Pontryagin numbers of which are given in Theorem \ref{Minumberthm}. With this information, it is straightforward to calculate the following $4$ particular characteristic numbers of these $4$ manifolds.
\begin{equation}\label{basiselevalueeq}
\begin{split}
&\widehat{A}(M_2)=1, \widehat{A}(M_i)=0, \ \ ~{\rm for}~i\neq 2,\\ 
&\widehat{A}(M_1, T)=-24, \widehat{A}(M_i, T)=0, \ \ ~{\rm for}~i\neq 1,\\ 
&\widehat{A}(M_1, \wedge^2)=1080, \widehat{A}(M_2, \wedge^2)=218076, \widehat{A}(M_3, \wedge^2)=-1, \widehat{A}(M_4, \wedge^2)=0, \\
&{\rm Sig}(M_1)=374784, {\rm Sig}(M_1)=378880, {\rm Sig}(M_1)=224, {\rm Sig}(M_1)=8.
\end{split}
\end{equation}
In \cite{MH02}, Mahowald-Hopkins showed that for any $24$ dimensional String manifold $M$
\begin{equation}\label{24achteq}
24~|~\widehat{A}(M,T)
\end{equation}
(this is also observed by Teichner \cite{Teichner}; cf. the discussions at Page $2961$ in \cite{CH15}). Together with Lemma \ref{sig8lemma}, there exists a well defined homomorphism of abelain groups
\begin{equation}\label{kappadefeq}
\kappa:=(\widehat{A}(-), \frac{1}{24}\widehat{A}(-,T), \widehat{A}(-, \wedge^2), \frac{1}{8}{\rm Sig}(-)): \Omega_{24}^{String}\rightarrow \mathbb{Z}\oplus \mathbb{Z}\oplus\mathbb{Z}\oplus\mathbb{Z}.
\end{equation}
The values of $\kappa$ (\ref{basiselevalueeq}) on the $4$ manifolds $M_i$ ($1\leq i \leq 4$) are given by the matrix
\begin{equation}\label{basismatrixeeqinproof}
K=(\kappa(M_1)^\tau, \kappa(M_2)^\tau, \kappa(M_3)^\tau, \kappa(M_4)^\tau)=
\begin{pmatrix}
0 & 1  & 0 & 0\\
-1 & 0 &0 & 0\\
1080  & 218076 & -1 & 0\\
46848& 47360 & 28 & 1
\end{pmatrix}
.
\end{equation}
It is clear that ${\rm det}(K)=-1$. In particular, $\kappa$ is an epimorphsim.
On the other hand, by the calculation of Gorbounov-Mahowald \cite{GM95}, it is known that
\[
 \Omega_{24}^{String}\cong \mathbb{Z}\oplus \mathbb{Z}\oplus\mathbb{Z}\oplus\mathbb{Z}.
\]
Hence, $\kappa$ is indeed an isomorphism, and $\{M_1, M_2, M_3, M_4\}$ is an integral basis of $\Omega_{24}^{String}$. We have showed Theorem \ref{basisthm} and Corollary \ref{basiscoro2}. For Corollary \ref{basiscoro1}, for any rational Pontryagin polynomial $P(-)$ we have 
\[
P(M)=a_1\widehat{A}(M)+a_2\frac{1}{24}\widehat{A}(M,T)+a_3\widehat{A}(M, \wedge^2)+a_4\frac{1}{8}{\rm Sig}(M),
\]
with some $a_i\in \mathbb{Q}$. If $P(M)\in \mathbb{Z}$ for any $M\in \Omega_{24}^{String}$. First choose $M=M_4$ and we have $P(M_4)=a_4$ by (\ref{basismatrixeeqinproof}). It follows that $a_4\in \mathbb{Z}$. Then choose $M=M_3$, and we have $P(M_3)=-a_3+28 a_4$ by (\ref{basismatrixeeqinproof}) which implies that $a_3\in \mathbb{Z}$. Finally, $P(M_2)=a_1+218076a_3+47360a_4$ implies that $a_1\in \mathbb{Z}$, while $P(M_1)=-a_2+1080 a_3+46848a_4$ implies that $a_2\in \mathbb{Z}$.
This completes the proof of Corollary \ref{basiscoro1}.
\hfill $\Box$

$\, $

\noindent{\it Proof of Theorem \ref{sigthm}.}\quad 
First let us make a comment on the condition $\omega_8(M)=0$. By Bensen-Wood \cite{BW95} or Duan \cite{Duan18}, it is known that 
\[
\frac{1}{2}p_1(M)=q_2(M)\equiv \omega_8(M) ~{\rm mod}~2.
\]
Hence the condition $\omega_8(M)=0$ is equivalent to $4~|~p_2(M)$. To show that $32~|~{\rm Sig}(M)$ under this condition, let us recall from Theorem \ref{Minumberthm} that the characteristic numbers $p_2^3$ of the basis manifolds $\{M_1, M_2, M_3, M_4\}$ are
\begin{equation}\label{p23m1-4eq}
(p_2^3(M_1),p_2^3(M_2),p_2^3(M_3),p_2^3(M_4))=(2^{13}\cdot 3^5\cdot 5^3, \ \ -2^{13}\cdot 3^5\cdot 5^3\cdot 41,  \ \  2^{7}\cdot 3^5\cdot 5,   \ \  2^4\cdot 3^5).
\end{equation}
By Theorem \ref{basisthm}, up to String cobrodism $M=\sum\limits_{i=1}^4 x_iM_i$ for some integral vector $(x_1, x_2, x_3, x_4)\in \mathbb{Z}^{\oplus 4}$. Hence 
\[
2^6~|~p_2^3(M)=\sum\limits_{i=1}^4 x_i p_2^3(M_i),
\]
which with (\ref{p23m1-4eq}) implies that 
\begin{equation}\label{4x4eq}
2^2~|~x_4.
\end{equation}
On the other hand, from (\ref{basiselevalueeq}) we have the signature vector
\begin{equation}\label{sigm1-4eq}
({\rm Sig}(M_1), {\rm Sig}(M_2), {\rm Sig}(M_3), {\rm Sig}(M_4))=(2^{11}\cdot 3\cdot 61,\ \  2^{11}\cdot 5\cdot 37,\ \  2^5\cdot 7,\ \  2^3).
\end{equation}
Combining it with (\ref{4x4eq}), it follows that $2^5~|~{\rm Sig}(M)$. The second statement in the theorem can be proved by the same strategy, and we have completed the proof Theorem \ref{sigthm}.
\hfill $\Box$

$\, $

\noindent{\it Proof of Theorem \ref{modsigthm}.}\quad 
By Hopkins-Singer \cite{HS05}, it can be computed that for $24$ dimensional String manifold $M$, the middle integral Spin Wu class
\[
\nu_{12}^{{\rm Spin}}(M)=5 p_3(M).
\]
Then with (\ref{sigm1-4eq}) and Theorem \ref{Minumberthm}, it is straightforward to compute the value of the modified signatures of the basis manifolds
\[
\begin{split}
({\rm Sig}(M_1,v), {\rm Sig}(M_2,v),& {\rm Sig}(M_3,v), {\rm Sig}(M_4,v))=\\
&(-2^{10}\cdot 3\cdot 826753, \ \ -2^{10}\cdot 5\cdot 23\cdot 668687, \ \ 2^5\cdot 7, \ \ -2^7\cdot 3\cdot 13),
\end{split}
\]
and the greatest common divisor
\[
{\rm g.c.d.}({\rm Sig}(M_1,v), {\rm Sig}(M_2,v), {\rm Sig}(M_3,v), {\rm Sig}(M_4,v))=32.
\]
The theorem then follows immediately from Theorem \ref{basisthm}.
Moreover, ${\rm Sig}(67 M_3+3M_4,v)=32$ by direct computation. This verifies the remark after Theorem \ref{modsigthm}.
\hfill $\Box$

$\, $

\noindent{\it Proof of Theorem \ref{twistedsigthm}.}\quad 
The theorem can be proved by the same strategy used in the proof of Theorem \ref{sigthm} with the value of the twisted signature of the manifolds
\begin{equation}\label{sigwedge2basiseq}
\begin{split}
({\rm Sig}(M_1,\wedge^2), {\rm Sig}(M_2,\wedge^2),& {\rm Sig}(M_3,\wedge^2), {\rm Sig}(M_4,\wedge^2))=\\
&(2^{13}\cdot 3\cdot 4013, \ \ -2^{13}\cdot 3^4\cdot 1063, \ \ 2^7\cdot 3\cdot 7\cdot 23, \ \ 2^5\cdot 3\cdot 23),
\end{split}
\end{equation}
which can be computed directly from Theorem \ref{Minumberthm}.
\hfill $\Box$
$\, $
\begin{remark}\label{sigwedge2remark}
Notice by (\ref{sigwedge2basiseq}) we also have 
\[
96~|~{\rm Sig}(M,\wedge^2)
\]
for any $24$-dimensional String manifold $M$. This reproves the result of Chen-Han \cite{CH15} that $3~|~{\rm Sig}(M, \wedge^2)$ by different methods.
\end{remark}

\section{$M_1$ and $M_2\in {\rm Coim}(W)$}\label{M12sec}
In this section, we review the information of two String manifolds $M_1$ and $M_2$ of dimension $24$ constructed by Mahowald-Hopkins \cite{MH02}. Let us start with Kervaire-Milnor's almost parallelizable manifolds. In \cite{MK58} Kervaire-Milnor showed that there is an almost parallelizable manifold $M_0^{4n}$ of dimension $4n$ with the top Pontryagin class
 \begin{equation}\label{m04nponeq}
 p_{n}(M_0^{4n})={\rm denom}(\frac{B_{2n}}{4n})\cdot a_n\cdot (2n-1)!\cdot x_{4n},
 \end{equation}
where $x_{4n}\in H^{4n}(M_0^{4n})$ is the generator,
\[
a_n=\left\{\begin{array}{ll}
2 & n={\rm odd}\\
1 & n={\rm even},
\end{array}
\right.
\]
and $B_{2n}$ is the Bernoulli number. Then it is easy to calculate that for $M_{0}^{4}$
\begin{equation}\label{m04}
p_1(M_{0}^{4})=48 x_4,\ \  {\rm Sig}(M_{0}^{4})=16,
\end{equation}
for $M_{0}^{8}$
\begin{equation}\label{m08}
p_2(M_{0}^{8})=1440 x_8,\ \  {\rm Sig}(M_{0}^{8})=224,
\end{equation}
and for $M_{0}^{12}$
\begin{equation}\label{m012}
p_3(M_{0}^{12})=120960 x_{12},\ \  {\rm Sig}(M_{0}^{12})=7936.
\end{equation}
The following proposition is well known.
\begin{proposition}\label{cobormhbais}
\[
\hspace{3.6cm}
\Omega_{\ast}^{SO}\otimes \mathbb{Q}\cong \mathbb{Q}[M_{0}^4, M_{0}^{8}, M_{0}^{12},\cdots].
\hspace{3.6cm}\Box
\]
\end{proposition}
From this proposition, Mahowald-Hopkins chose a particular basis for $\Omega_{24}^{SO}\otimes \mathbb{Q}$
\begin{equation}\label{mhbasisqstringcob}
\begin{split}
&B_1=M_{0}^8\times M_{0}^8\times M_{0}^8,\\
&B_2=\frac{1}{2} M_{0}^{12}\times \frac{1}{2} M_{0}^{12},\\
&B_3=M_{0}^8\times M_{0}^{16},\\
&B_4=\frac{1}{2} M_{0}^{24}.
\end{split}
\end{equation}
They called $\frac{1}{2} M_{0}^{8k+4}$ a {\it fake manifold} since it is not a proper manifold. Nevertheless, they showed that there is a proper manifold $B_2$ with its Pontryagin numbers equal to those of the square of  $\frac{1}{2} M_{0}^{12}$. Among others, in \cite{MH02} Mahawold-Hopkins determined the image of Witten genus at dimension $24$. Recall that at this particular case, there is the famous formula of Hirzebruch (Page $85$-$87$ in \cite{HBJ94})
\begin{equation}
W(M)=\widehat{A}(M)\bar{\Delta}+\widehat{A}(M,T)\Delta,
\end{equation}
for any $24$ dimensional String manifold $M$, where 
\[
\bar{\Delta}=E_4^3-744\cdot \Delta,
\]
with $E_4$ the Eisenstein series of weight $4$ and $\Delta$ the famous cusp form of weight $12$.
\begin{theorem}[Section $9$ in \cite{MH02}]\label{MHthm}
There exist two proper String manifold $M_1$ and $M_2$ of dimension $24$, such that in the rational oriented cobordism ring
\begin{equation}\label{m12def}
M_1=\frac{B_1+B_2}{72}, \ \ \ M_2=\frac{-41B_1+31B_2}{72}.
\end{equation}
Furthermore, the image of Witten genus at dimension $24$
\begin{equation}\label{imagewitteneq}
{\rm Im}\big\{  W: \Omega_{24}^{String}\rightarrow \mathbb{Z}[[q]]\big\}\cong \mathbb{Z}\{M_1, M_2\},
\end{equation}
with 
\begin{equation}\label{valuewm1m2}
\hspace{3.6cm}
W(M_1)=-24\Delta, \  \ \  W(M_2)=\bar{\Delta}. \hspace{3.6cm}\Box
\end{equation}
\end{theorem}
Let us summarize the Pontryagin numbers of $M_1$ and $M_2$.
\begin{lemma}\label{ponnom1m2lemma}
For $M_1$,
\[
p_2^3=2^{13}\cdot 3^5\cdot 5^3, \ \ p_3^2=2^{10}\cdot 3^4\cdot 5^2 \cdot 7^2, \ \
p_2p_4=2^{12}\cdot 3^5\cdot 5^3, \ \ p_6=2^9\cdot 3^4\cdot 5^2\cdot 89,
\]
and for $M_2$,
\[
p_2^3=-2^{13}\cdot 3^5\cdot 5^3\cdot 41, \ \ p_3^2=2^{10}\cdot 3^4\cdot 5^2 \cdot 7^2\cdot 31, \ \
p_2p_4=-2^{12}\cdot 3^5\cdot 5^3\cdot 41, \ \ p_6=-2^9\cdot 3^4\cdot 5^2\cdot 11^2. \]
~$\qqed$
\end{lemma}


\section{$M_3\in {\rm Ker}(W)$}\label{M3sec}
Since the image of the Witten genus is known, we are left to consider its kernel. There is an outstanding principle to attack it.
\begin{theorem}[Jung and Dessai; \cite{Des94}]
The ideal of $\Omega_{\ast}^{String}\otimes \mathbb{Q}$, consisting of bordism classes of Caley plane bundles with connected structure groups, is precisely the kernel of the rational Witten genus. ~$\qqed$
\end{theorem} 
The local version of the theorem was proved by McTague \cite{McTague14} for the localization of Witten genus away from $6$. The simplest Caley plane bundles are, of course, the trivial ones. Let us define
\begin{equation}\label{m3def}
M_3:=M_{0}^{8}\times \mathbb{O}P^2,
\end{equation}
where $M_{0}^{8}$ is the almost parallelizable manifold (\ref{m04nponeq}) of dimension $8$, and $\mathbb{O}P^2$ is {\it the Caley plane or the octonionic projective plane}. The cell structure of $\mathbb{O}P^2$ is clear from its cohomology ring 
\begin{equation}\label{cohomop2eq}
H^\ast(\mathbb{O}P^2)\cong \mathbb{Z}[u_{8}]/u_{8}^3,
\end{equation}
where ${\rm deg}(u_8)=8$.
Further, $\mathbb{O}P^2$ is a $16$ dimensional manifold with the total Pontryagin class (Theorem $19.4$ in \cite{BH58})
\begin{equation}\label{ponop2eq}
p(\mathbb{O}P^2)=1+6u_8+39u_8^2.
\end{equation}
Hence, we can compute all the Pontryagin classes of $M_3$. It is clear that
\[
W(M_3)=0,
\]
but there is the particular twisted genus 
\begin{equation}\label{ahatwedge2m3=-1eq}
\widehat{A}(M_3, \wedge^2)=-1.
\end{equation}
Let us summarize the Pontryagin numbers of $M_3$.
\begin{lemma}\label{ponnom3lemma}
For $M_3$,
\[
\hspace{2cm}
p_2^3=2^{7}\cdot 3^5\cdot 5, \ \ p_3^2=0, \ \
p_2p_4=2^{5}\cdot 3^3\cdot 5^3, \ \ p_6=2^5\cdot 3^3\cdot 5\cdot 13.
\hspace{2cm}\Box
\]
\end{lemma}


\section{$M_4\in {\rm Ker}(W)$}\label{M4sec}
We continue to construct particular String manifolds of dimension $24$ in the kernel of Witten genus. In this non-trivial case, we need to construct an appropriate closed $8$-manifold $N^8$ and apply the pullback diagram
\begin{gather}
\begin{aligned}
\xymatrix{
\mathbb{O}P^2 \ar@{=}[r]  \ar[d]  & \mathbb{O}P^2 \ar[d] \\
M^{24} \ar[r]^<<<<<{\tilde{f}} \ar[d]^{\pi}   & BSpin(9)\ar[d]^{\Theta}\\
N^8 \ar[r]^<<<<<<<{f}   & BF_4,
}
\end{aligned}
\label{f4op2bundledia}
\end{gather}
where $F_4$ is the exceptional Lie group, and
\begin{equation}\label{unif4op2bundleeq}
\mathbb{O}P^2\rightarrow BSpin(9)\stackrel{\Theta}{\rightarrow} BF_4
\end{equation}
is called {\it the universal $F_4$-$\mathbb{O}P^2$-bundle} following Klaus \cite{Klaus95}. 
It exists since $Spin(9)$ is the subgroup of $F_4$ with the quotient
\[
F_4/Spin(9)\cong \mathbb{O}P^2.
\]
The pullback bundle $\pi$ is called {\it an $F_4$-$\mathbb{O}P^2$-bundle}, as a generalization of {\it $PSp(3)$-$\mathbb{H}P^2$-bundles} of Kreck-Stolz \cite{KS93}. These bundles were studied by Borel-Hirzebruch \cite{BH58} in general context. In particular, Borel-Hirzebruch \cite{BH58} developed a theory with associated algorithm to compute the Pontrygin classes of such bundles. In the following, we will first recall the Borel-Hirzebruch algorithm, and then construct an appropriate $M_4$ step by step.

\subsection{Borel-Hirzebruch algorithm}
Given any fibre bundle 
\begin{equation}\label{inputbundleeq}
F\rightarrow E\stackrel{p}{\rightarrow} B
\end{equation}
with structural group $G$, and $F$, $E$, $B$ are all manifolds. Set ${\rm dim}F=n$. There is the induced bundle 
\begin{equation}\label{bundlefibreeq}
\mathbb{R}^n\rightarrow E\times_{G} TF\rightarrow E,
\end{equation}
where the action of $G$ on the tangent bundle $TF$ is induced from that on $F$. The bundle (\ref{bundlefibreeq}), denoted by $p^{\Delta}$ as in \cite{Klaus95}, is called {\it the bundle along the fibre associated to the bundle $p$} (\ref{inputbundleeq}).
In particular, it is easy to see that 
\begin{equation}\label{tangentdecomeq}
TE\cong p^{\Delta}\oplus p^\ast(TB).
\end{equation}

Now let $G$ be a compact connected Lie group with subgroup $H$. The principal bundle
\[
H\rightarrow G\rightarrow G/H
\]
can be extended twice to the right, and we have the fibre bundle 
\begin{equation}\label{unighbundleeq}
G/H\rightarrow BH \stackrel{\Theta}{\rightarrow} BG
\end{equation}
Let $S$ be the maximal torus of $H$. The inclusion of the maximal torus induces a map of classifying maps
\begin{equation}\label{bshrhoeq}
\rho: BS\rightarrow BH
\end{equation}
\begin{theorem}[Special version of Theorem $10.7$ in \cite{BH58}; the universal case]\label{BHthm}
Let $S\leq H\leq G$ and $\rho$ as above. Denote by
\[
\{\pm b_j\}_{j=1}^{k}
\]
the set of the roots of $G$ with respect to $S$, which are complementary to those of $H$ (view $b_j\in H^2(BS;\mathbb{Z})$). Then the Pontryagin class of the bundle along the fibre $\Theta^{\Delta}$, associated to the fibre bundle $\Theta$, is determined by 
\begin{equation}\label{BHalgeq}
\hspace{4.4cm}
\rho^\ast(p(\Theta^{\Delta}))=\prod_{j=1}^{k}(1+b_j^2).
\hspace{4.4cm}\Box
\end{equation}
\end{theorem}

\subsection{Step $1$: compute the Pontryagin class of the bundle along the fibre $\Theta^{\Delta}$}\label{step1subsec}
Now let us restrict ourselves to consider the bundle (\ref{unif4op2bundleeq})
\[
\mathbb{O}P^2\rightarrow BSpin(9)\stackrel{\Theta}{\rightarrow} BF_4.
\]
Recall that $F_4$ and $Spin(9)$ are of both rank $4$, and there is a maximal torus 
\[
S\cong T^4\hookrightarrow Spin(9),
\]
which is also the maximal torus of $F_4$ via the inclusion $Spin(9)\hookrightarrow F_4$.
Denote
\[
H^2(BS;\mathbb{Z})\cong \mathbb{Z}[x_1,x_2,x_3,x_4].
\]
It is known that the roots of $Spin(9)$, with respect to $S$, are
\[
\pm x_i \pm x_j ~(1\leq j<j\leq 4),~~\ \ \pm x_1,\pm x_2,\pm x_3,\pm x_4,
\]
while the complementary root of $F_4$ are
\[
\frac{1}{2}(\pm x_1\pm x_2\pm x_3\pm x_4).
\]
Let $r_i=\frac{1}{2}(x_1\pm x_2\pm x_3\pm x_4)$. By Theorem \ref{BHthm}, we have that
\begin{equation}\label{bhponthetaeq}
\rho^\ast(p(\Theta^{\Delta}))=\prod_{i=1}^{8}(1+r_i^8),
\end{equation}
where $\rho: BS\rightarrow BSpin(9)$. On the other hand, we know that
\[
\prod_{i=1}^4(1+x_i^2)=\rho^\ast(1+p_1+p_2+p_3+p_4),
\]
where $p_i\in H^{4i}(BSpin(9)$ is the $i$-th Pontryagin class. Hence, by straightforward calculation we obtain the following
\begin{proposition}\label{op2f4fibreponprop}
The Pontryagin class of the bundle along the fibre associated to the universal $F_4$-$\mathbb{O}P^2$-bundle $\Theta$ is
\begin{equation}\label{unifibreponeq}
\begin{split}
p(\Theta^{\Delta})=&1+(2p_1)
+(-p_2+\frac{7}{4}p_1^2)
+(2p_3-\frac{3}{2}p_1p_2+\frac{7}{8}p_1^3)\\
&+(-\frac{17}{2}p_4+2p_1p_3+\frac{3}{8}p_2^2-\frac{15}{16}p_1^2p_2+\frac{35}{128}p_1^4)\\
&+(-\frac{5}{2}p_1p_4-p_2p_3+\frac{3}{4}p_1^2p_3+\frac{3}{8}p_1p_2^2-\frac{5}{16}p_1^3p_2+\frac{7}{128}p_1^5)\\
&+(-\frac{7}{4}p_2p_4+\frac{5}{16}p_1^2p_4+p_3^2-\frac{1}{2}p_1p_2p_3+\frac{1}{8}p_1^3p_3-\frac{1}{16}p_2^3\\
&~~ \ \ \ +\frac{9}{64}p_1^2p_2^2-\frac{15}{256}p_1^4p_2+\frac{7}{1024}p_1^6).
\end{split}
\end{equation}
~$\qqed$
\end{proposition}

\subsection{Step $2 $: the appropriate base manifold and classifying map $(N^8,f)$}
At this step, we construct an appropriate Spin manifold $N^8$ of dimension $8$ as the base manifold of the $F_4$-$\mathbb{O}P^2$-bundle $\pi$ in (\ref{f4op2bundledia}). We need Wall's $(n-1)$-connected $2n$-manifolds with $n=4$ \cite{Wall62} (also see \cite{Duan18}). 
\begin{definition}\label{wallpairdef}
Let $A=\{a_{ij}\}_{n\times n}$ be a unimodular symmetric integral matrix of rank $n$, $b=(b_1,b_2,\cdots, b_n)$ be a sequence of integers of length $n$. The pair $(A, b)$ is called a {\it Wall pair} if it satisfies the congruent condition
\begin{equation}\label{wallcond1}
a_{ii}\equiv b_i~{\rm mod}~2, \ \ 1\leq i\leq n.
\end{equation}
\end{definition}
It is natural to ask that which Wall pairs $(A,b)$ can be realized as the pair $(I(N^8), q_1(N^8))$ of a Wall manifold $N^8$; here $I(N^8)$ is the intersection form of $N^8$ and $q_1(N^8)=\frac{1}{2}p_1(N^8)$ is the first Spin class of $N^8$ (also see (\ref{spin9cohmeq}) and (\ref{pqforeq})).
\begin{theorem}[Theorem $4$ in \cite{Wall62}; also see Theorem $10.11$ and $10.13$ in \cite{Duan18}]\label{Wall8thm}
For any Wall pair $(A,b)$ such that
\begin{equation}\label{wallcond2}
{\rm Sig}(A)\equiv bAb^\tau~{\rm mod}~224,
\end{equation}
there exists a smooth manifold $N^8$ such that under a certain choice of basis of
\[
H^4(N^8;\mathbb{Z})\cong \mathbb{Z}\{x_1,\cdots, x_n\},
\]
the intersection form $I(N^8)$ is represented by the matrix $A$, and the first Spin class $q_1(N^8)$ is represented by $b$; in other word,
\[
I(N^8)(x_i, x_j)=a_{ij}, \ \ ~{\rm and}~ q_1(N_8)=b_1x_1+\cdots b_nx_n.
\]
\end{theorem}
\begin{proof}
The theorem has been proved in \cite{Duan18} based on \cite{Wall62}. Indeed Wall \cite{Wall62} showed that for each Wall pair $(A,b)$ there exists a closed $8$ dimensional topological manifold $N^8$ such that its intersection form is represented by $A$ and its first Spin class is represented by $b$. Moreover, $M=W\cup_h D^8$ where $W$ is a $3$-connected smooth manifold with boundary $\partial W$ a homotopy $7$-sphere and $h:\partial W\rightarrow \partial D^8$ is a homeomorphism. To show that $N^8$ is smooth, we may compute the Eells-Kuiper $\mu$-invariant \cite{EK62} of the boundary $\partial W$
\[
\mu(\partial W)\equiv\frac{bAb^\tau-{\rm Sig}(A)}{224} ~{\rm mod}~1.
\]
Since $\mu$-invariant is a complete invariant for homotopy $7$-spheres, we see that $\partial W$ is diffeomorphic to the standard $S^7$, and hence $N^8$ is smooth.
\end{proof}

We now apply Theorem \ref{Wall8thm} to construct an appropriate $N^8$ such that after particular pullback $f$ the total space $M^{24}$ in Diagram (\ref{f4op2bundledia}) will be a String manifold. 
For that, we may choose 
\begin{equation}\label{wallpairm4eq}
A={\rm diag}(H, E_8),\ \ \  b=(2, 2, 0, \cdots, 0)
\end{equation}
where 
\[
H=
\begin{pmatrix}
0 & 1\\
1 & 0
\end{pmatrix}
,~{\rm and}~ \ \ 
E_8=
\begin{pmatrix}
2 &1 &0 &0 &0 &0 &0 &0 \\
1 &2 &1 &0 &0 &0 &0 &0 \\
0 &1 &2 &1 &0 &0 &0 &0 \\
0 &0 &1 &2 &1 &0 &0 &0 \\
0 &0 &0 &1 &2 &1 &0 &1 \\
0 &0 &0 &0 &1 &2 &1 &0 \\
0 &0 &0 &0 &0 &1 &2 &0 \\
0 &0 &0 &0 &1 &0 &0 &2 
\end{pmatrix}
\]
are the hyperbolic matrix of rank $2$ and the Cartan matrix of the exceptional Lie group $E_8$ respectively. It is then clear that the conditions (\ref{wallcond1}) and (\ref{wallcond2}) are satisfied, and even better
\begin{equation}\label{Wallpairm4sigeq}
{\rm Sig}(A)=bAb^{\tau}=8.
\end{equation}
Hence, by Theorem \ref{Wall8thm} there exists a smooth $N^8$ such that 
\begin{equation}\label{h4n8eq}
H^4(N^8)\cong \mathbb{Z}\{a_1,a_2, b_1,\cdots b_8\},
\end{equation}
\begin{equation}\label{m4q1eq}
q_1(N^8)=2(a_1+a_2),
\end{equation}
and under the basis $\{a_1,a_2, b_1,\cdots b_8\}$ the intersection form of $N^8$ is represented by $A$ in (\ref{wallpairm4eq}).
In particular, we can use the Hirzebruch signature formula to calculate the second Pontryagin class of $N^8$.
\begin{lemma}\label{ponn8lemma}
\[
\hspace{3.6cm}
p(N^8)=1+4(a_1+a_2)+56a_1a_2.
\hspace{3.6cm}\Box
\]
\end{lemma}

In Diagram (\ref{f4op2bundledia}), by Lemma \ref{f4bundleclasslemma} below let us choose 
\[
f: N^8\rightarrow BF_4
\]
such that 
\begin{equation}\label{fdefeq}
f^\ast(x_4)=-(a_1+a_2),
\end{equation}
where $x_4\in H^4(BF_4)$ is the generator such that (cf. (\ref{x4theta}))
\begin{equation}\label{x4q1eq}
\Theta^\ast(x_4)=q_1\in H^4(BSpin(9)).
\end{equation}
We notice that by Proposition \ref{op2f4fibreponprop}
\[
p_1(\Theta^{\Delta})=2p_1=4q_1\in H^4(BSpin(9))
\]
Hence, by (\ref{tangentdecomeq}) we have
\[
\begin{split}
p_1(M^{24})
&=p_1(\pi^{\Delta})+\pi^\ast(p_1(N^8))\\
&=\tilde{f}^\ast(p_1(\Theta^{\Delta}))+4(a_1+a_2)\\
&=4\tilde{f}^\ast(q_1)+4(a_1+a_2)\\
&=4\tilde{f}^\ast\circ \Theta^\ast(x_4)+4(a_1+a_2)\\
&=4 \pi^\ast \circ f^\ast(q_1)+4(a_1+a_2)\\
&=0.
\end{split}
\]
Hence, $M^{24}$ is a String manifold, and from now on we may denote this particular String manifold by $M_4$.

\begin{lemma}\label{f4bundleclasslemma}
There is a natural isomorphism of sets
\[
[N^8, BF_4]\cong [\bigvee_{i=1}^{2}S_{a_i}^4\vee\bigvee_{j=1}^{8}S_{b_j}^4, BF_4]\cong \mathbb{Z}^{\oplus 10},
\]
where $S_{a_i}^4$ ($i=1$, $2$) and $S_{b_j}^4$ ($1\leq j\leq 8$) represents the cohomology class $a_i$ and $b_j$ in (\ref{h4n8eq}) respectively. 
\end{lemma}
\begin{proof}
By the computation of Mimura \cite{Mimura67}, it is known that 
\[
\pi_i(BF_4)=0, \ \ \ 0 \leq i \leq 8, ~{\rm and}~i\neq 4.
\]
Then by applying the functor $[-, BF_4]$ to the cofibre sequence determined the attaching map of $N^8$, we get an exact sequence. From that the lemma follows easily.
\end{proof}

\subsection{Step $3$: determine the pullback image of $H^\ast(BSpin(9))$}
In the last step, we have constructed the String manifold $M_4$ as the total space of the $F_4$-$\mathbb{O}P^2$-bundle over the particular Wall manifold $N^8$, via the pullback diagram
\begin{gather}
\begin{aligned}
\xymatrix{
\mathbb{O}P^2 \ar@{=}[r]  \ar[d]  & \mathbb{O}P^2 \ar[d] \\
M_4 \ar[r]^<<<<<{\tilde{f}} \ar[d]^{\pi}   & BSpin(9)\ar[d]^{\Theta}\\
N^8 \ar[r]^<<<<<<<{f}   & BF_4,
}
\end{aligned}
\label{f4op2bundle2dia}
\end{gather}
such that 
\[
f^\ast(x_4)=-(a_1+a_2).
\]
It is clear that 
\begin{equation}\label{cohm4eq}
H^\ast(M_4)\cong H^\ast(N)[u_8]/\langle u_8^3-t a_1 a_2 u_8^2 \rangle,
\end{equation}
for some $t\in \mathbb{Z}$, and $a_1 a_2 u_8^2\in H^{24}(M_4)$ is a generator.
In order to compute the Pontryagin class of $M_4$, we need to determine the image of $H^\ast(BSpin(9))$ under $\tilde{f}^\ast$.

First, by the computation of Duan \cite{Duan18} it is known that (cf. Thomas \cite{Thomas62} and Benson-Wood \cite{BW95})
\begin{equation}\label{spin9cohmeq}
H^\ast(BSpin(9))\cong \mathbb{Z}[q_1, q_2, q_3, q_4]\oplus({\rm the}~2~{\rm torsion}~{\rm part}),
\end{equation}
where $q_i$ is called the {\it the $i$-th universal Spin class} with ${\rm deg}(q_i)=4i$. The Spin classes determine the Pontryagin classes, in which way they illustrate the divisibility of Pontryagin classes of Spin manifolds. In the low dimensions, the conversion formulae are
\be\label{pqforeq}
\begin{split}
&p_1=2q_1,\\
&p_2=2q_2+q_1^2,\\
&p_3=q_3,\\
&p_4=2q_4+q_2^2-2q_1q_3.
\end{split}
\ee
On the other hand, it is also known that 
\begin{equation}\label{f4cohmeq}
H^\ast(BF_4)\cong \mathbb{Z}[x_4, x_{12}, x_{16}]\oplus({\rm the}~{\rm torsion}~{\rm part}),
\end{equation}
where ${\rm deg}(x_i)=i$.
Since  $\mathbb{O}P^2$ is $7$-connected and $BF_4$ is $3$-connected, the fibre bundle 
\[
\mathbb{O}P^2\stackrel{i}{\rightarrow} BSpin(9)\stackrel{\Theta}{\rightarrow} BF_4
\]
is a cofibre sequence up to degree $11$, by the dual Blakers-Massey theorem or a simple argument of the Serre spectral sequence. In particular, we have 
\begin{equation}\label{x4theta}
\Theta^\ast(x_4)=q_1,
\end{equation}
and there is an exact sequence 
\[
\begin{split}
\rightarrow 0=H^7(\mathbb{O}P^2)& \rightarrow H^8(BF_4)\stackrel{\Theta^\ast}{\rightarrow} H^8(BSpin(9))\\ 
&\stackrel{i^\ast}{\rightarrow} H^8(\mathbb{O}P^2)\rightarrow H^9(BF_4)\rightarrow H^9(BSpin(9))=0.
\end{split}
\]
Since $\Theta^\ast$ maps $H^8(BF_4)\cong \mathbb{Z}\{x_4^2\}$ isomorphically onto $\mathbb{Z}\{q_1^2\}\leq H^8(BSpin(9))$ and $H^9(BF_4)\cong \mathbb{Z}/3$ by Toda \cite{Toda73}, the above exact sequence implies the short exact sequence
\[
0\rightarrow\mathbb{Z}\{q_2\}\rightarrow  H^8(\mathbb{O}P^2)\cong\mathbb{Z}\{u_8\}\rightarrow \mathbb{Z}/3\rightarrow 0.
\]
Hence 
\begin{equation}
i^\ast (q_2)=3u_8,
\end{equation}
which implies that 
\begin{equation}\label{tildefq2}
\tilde{f}^\ast (q_2)=3u_8+k a_1a_2,
\end{equation}
for some $k\in \mathbb{Z}$.
In order to determine the image of $q_3$ and $q_4$ under $\tilde{f}$, we need to use the Weyl invariants of $F_4$.
\begin{theorem}[Borel \cite{Borel55}]\label{Borelthm}
 Let $G$ be a compact Lie group with a maximal torus $T$ and Weyl group $W_G$. The inclusion $T\hookrightarrow G$ induces an isomorphism
 \[
  \hspace{3.9cm}
 H^\ast(BG;\mathbb{Q})\cong H^\ast(BT;\mathbb{Q})^{W_G}.
 \hspace{3.9cm}\Box
 \]
\end{theorem}
We borrow the notations from \hyperref[step1subsec]{Step $1$}.
 Let
\[
\prod_{i=1}^4(1+x_i^2)=1+p_1+p_2+p_3+p_4,
\]
and 
\[
r_i=\frac{1}{2}(x_1\pm x_2\pm x_3\pm x_4), \ \ 1\leq i\leq 8.
\]
Let 
\[
I_{2k}=\sum\limits_{i=1}^{4}x_{i}^{2k}+\sum\limits_{j=1}^{8}r_j^{2k}.
\]
It is known that (for instance, see \cite{Mehta88}, or \cite{Tshishiku15})
\[
H^\ast(BS;\mathbb{Q})^{W_{F_4}}=\mathbb{Q}[I_2, I_6, I_8, I_{12}].
\]
From this, it is not hard to show that (Section $19$ in \cite{BH58})
\begin{equation}\label{weylf4eq}
H^{\leq 16}(BS;\mathbb{Q})^{W_{F_4}}\cong \mathbb{Q}^{\leq 16}[p_1, -6p_3+p_1p_2, 12p_4+p_2^2-\frac{1}{2}p_1^2p_2].
\end{equation}
\begin{lemma}\label{forq3q3imagelemma}
\[
\tilde{f}^\ast(-6p_3+p_1p_2)=0, \ \ \ \tilde{f}^\ast(12p_4+p_2^2-\frac{1}{2}p_1^2p_2)=0.
\]
\end{lemma}
\begin{proof}
We have the commutative diagram
\begin{gather}
\begin{aligned}
\xymatrix{
H^\ast(M_4) 
&H^\ast(BSpin(9))\ar[r]^{\rho^\ast}  \ar[l]_{\tilde{f}} 
& H^\ast(BS)^{W_{Spin(9)}}\\
H^\ast(N^8) \ar[u]^{\pi^\ast}
&H^\ast(BF_4) \ar[r]^{\rho^\ast} \ar[u]^{\Theta^\ast} \ar[l]_{f^\ast}
& H^\ast(BS)^{W_{F_4}},\ar[u]
}
\end{aligned}
\label{f4spin9weyldia}
\end{gather}
which particularly implies that $\tilde{f}^\ast\circ \Theta^\ast(x)=0$ for any $x$ with ${\rm deg}(x)>8$.
Then by Theorem \ref{Borelthm} and (\ref{weylf4eq}), the lemma follows easily.
\end{proof}
We can now determine the pullback image of $H^\ast(BSpin(9))$ through $\tilde{f}$.
\begin{lemma}\label{qimagelemmaq}
\[
\begin{split}
&\tilde{f}^\ast(q_1)=-(a_1+a_2),\\
&\tilde{f}^\ast(q_2)=3u_8+ka_1a_2,\\
&\tilde{f}^\ast(q_3)=-2(a_1+a_2)u_8,\\
&\tilde{f}^\ast(q_4)=-6u_8^2+4a_1a_2u_8-4ka_1a_2 u_8.
\end{split}
\]
\end{lemma}
\begin{proof}
The image of $q_1$ and $q_2$ were determined already. For the other two, we only need to use the conversion formulas (\ref{pqforeq}) to rewrite the two equalities in Lemma \ref{forq3q3imagelemma} in Spin classes, and then solve $\tilde{f}^\ast(q_3)$ and $\tilde{f}^\ast(q_4)$ from them directly.
\end{proof}
\subsection{Step $4$: compute the Pontryagin numbers of $M_4$}
We are now in a position to compute the Pontryagin numbers of $M_4$. First, let us translate the image of $H^\ast(BSpin(9))$ under $\tilde{f}$, obtained in Lemma \ref{qimagelemmaq}, in terms of Pontryagin classes by using the conversion formulas (\ref{pqforeq}).
\begin{lemma}\label{qimagelemmap}
\[
\begin{split}
&\hspace{4.1cm}\tilde{f}^\ast(p_1)=-2(a_1+a_2),\\
&\hspace{4.1cm}\tilde{f}^\ast(p_2)=6u_8+2(k+1)a_1a_2,\\
&\hspace{4.1cm}\tilde{f}^\ast(p_3)=-2(a_1+a_2)u_8,\\
&\hspace{4.1cm} \tilde{f}^\ast(p_4)=-3u_8^2-2ka_1a_2u_8. \hspace{4.1cm}\Box
\end{split}
\]
\end{lemma}
By (\ref{tangentdecomeq}), we know that
\begin{equation}\label{calponm4eq}
p(M_4)=\pi^\ast(p(N^8))\cdot \tilde{f}^\ast(p(\Theta^\Delta)).
\end{equation}
With Lemma \ref{ponn8lemma} for $p(N^8)$, Proposition \ref{unifibreponeq} for $p(\Theta^\Delta)$ and Lemma \ref{qimagelemmap} for $\tilde{f}$, it is now straightforward to calculate the Pontryagin class of $M_4$. 
\begin{lemma}\label{ponm4lemma}
\[
\hspace{1cm}
\begin{split}
p(M_4)=
&1+(36a_1a_2-2ka_1a_2-6u_8)-10(a_1+a_2)u_8\\
&+(-244a_1a_2+26ka_1a_2+39u_8)u_8\\
&+126(a_1+a_2)u_8^2+(1958a_1a_2+18ka_1a_2+18u_8)u_8^2. 
\hspace{1cm}\Box
\end{split}
\]
\end{lemma}
We can now determine the Pontryagin numbers of $M_4$.
\begin{lemma}\label{ponnom4lemma}
For $M_4$,
\[
p_2^3=3888, \ \ p_3^2=200, \ \
p_2p_4=2868, \ \ p_6=1958.
\]
\end{lemma}
\begin{proof}
Recall by (\ref{cohm4eq}) $u_8^3=t a_1 a_2 u_8^2$ and $a_1 a_2 u_8^2$ is a generator of $H^{24}(M)$.
By Lemma \ref{ponm4lemma}, it is straightforward to calculate that
\begin{equation}\label{m4pno.kt}
p_2^3=3888-216(k+t), \ \ p_3^2=200, \ \
p_2p_4=2868-234(k+t),\ \ p_6=1958+18(k+t).
\end{equation}
Hence, by Hirzebruch's signature theorem, it is easy to calculate that 
\begin{equation}\label{sigm4kt}
{\rm Sig}(M_4)=8+\frac{9590138}{70945875}(k+t).
\end{equation}
On the other hand, by a theorem of Chern-Hirzebruch-Serre \cite{CHS57} we know that
\begin{equation}\label{sigm4=8}
{\rm Sig}(M_4)={\rm Sig}(N^8){\rm Sig}(\mathbb{O}P^2)=8.
\end{equation}
Hence from (\ref{sigm4kt}) $k+t=0$, and the lemma follows from (\ref{m4pno.kt}).
\end{proof}

\appendix
\section{Twisted $A$-hats and twisted signatures}\label{backsec}
Let $M$ be a $4m$ dimensional oriented closed smooth manifold. There are two important characteristic numbers, namely {\it the (twisted) $A$-hat genus} and {\it the (twisted) signature}, which are the topological pillars of the Atiyah-Singer index theory. 

Equip $M$ with a Riemannian metric $g^{TM}$. Let $\nabla ^{TM}$ be the
associated Levi-Civita connection on $TM$ and $R^{TM}=(\nabla ^{TM})^{2}$ be
the curvature of $\nabla ^{TM}$. $\nabla ^{TM}$ extends canonically to a
Hermitian connection $\nabla ^{T_{\mathbf{C}}M}$ on $T_{\mathbf{C}%
}M=TM\otimes \mathbf{C}$, the complexification of $TM$. 

Let $\widehat{A}(TM,\nabla ^{TM})$ be the
Hirzebruch $\widehat{A}$-form defined by (cf. \cite{Z})
\begin{equation}
\widehat{A}(TM,\nabla ^{TM})={\det }^{1/2}\left( {\frac{{\frac{\sqrt{-1}}{%
4\pi }}R^{TM}}{\sinh \left( {\frac{\sqrt{-1}}{4\pi }}R^{TM}\right) }}\right).
\end{equation}

Let $E$ be a Hermitian vector bundles over $M$ carrying a Hermitian
connection $\nabla ^{E}$. Let $R^{E}=(\nabla
^{E})^{2}$ be the curvature of $\nabla
^{E} $. The Chern character form (cf. \cite{Z}) is defined as
\begin{equation}
\mathrm{ch}(E,\nabla ^{E})=\mathrm{tr}\left[ \exp \left( {\frac{\sqrt{-1}}{%
2\pi }}R^{E}\right) \right].
\end{equation}

The {\em $\widehat{A}$-genus} and the {\em twisted $\widehat{A}$-genus} are defined respectively as 
\begin{equation}
\begin{split}
& \widehat{A}(M)=\int_M \widehat{A}(TM,\nabla ^{TM}), \\
& \widehat{A}(M,E)=\int_M \widehat{A}(TM,\nabla ^{TM})\mathrm{ch}(E,\nabla ^{E}).
\end{split}%
\end{equation}

When $M$ is spin, let $S(TM)=S_{+}(TM)\oplus S_{-}(TM)$ denote the bundle of
complex spinors associated to the Spin structure. Then $S(TM)$ carries induced Hermitian metric and
connection preserving the above ${\bf Z}_2$-grading. Let
$$D_{\pm}:\Gamma(S_{\pm}(TM))\rightarrow \Gamma(S_{\mp}(TM))$$
denote the induced Spin Dirac operators (cf. \cite{Law}). By the Atiyah-Singer index theorem, 
\begin{equation}
\begin{split}
& \widehat{A}(M)=\ind(D), \\
& \widehat{A}(M,E)=\ind(D\otimes E).
\end{split}%
\end{equation}

Let $\widehat{L}(TM,\nabla ^{TM})$ be the
Hirzebruch characteristic form defined by (cf. \cite{Liu95}, \cite{Z})
\begin{equation}
\widehat{L}(TM,\nabla ^{TM})={\det }^{1/2}\left( {\frac{{\frac{\sqrt{-1}}{%
2\pi }}R^{TM}}{\tanh \left( {\frac{\sqrt{-1}}{4\pi }}R^{TM}\right) }}\right).
\end{equation}
Note that $\widehat{L}(TM,\nabla ^{TM})$ defined here is different from the classical Hirzebruch $L$-form defined by
\begin{equation*}
L(TM,\nabla ^{TM})={\det }^{1/2}\left( {\frac{{\frac{\sqrt{-1}}{%
2\pi }}R^{TM}}{\tanh \left( {\frac{\sqrt{-1}}{2\pi }}R^{TM}\right) }}\right).
\end{equation*} However they give same top (degree $4m$) forms and therefore 
\be \int_M\widehat{L}(TM,\nabla ^{TM}) =\int_M L(TM,\nabla ^{TM}).\ee
We would also like to point out that our $\widehat{L}$ is different from the ${\bf \widehat{L}}$ in page 233 of \cite{Law}.

Let $\mathrm{ch}(E,\nabla ^{E})=\sum_{i=0}^{2m}\mathrm{ch}^{i}(E,\nabla ^{E})$ such that $\mathrm{ch}^{i}(E,\nabla ^{E})$ is the degree $2i$ component. Define 
\be \mathrm{ch}_2(E,\nabla ^{E})=\sum_{i=0}^{2m}2^i\mathrm{ch}^{i}(E,\nabla ^{E}).\ee
It's not hard to see that 
\be \int_M \widehat{L}(TM,\nabla ^{TM}) \mathrm{ch}(E,\nabla ^{E})=\int_M L(TM,\nabla ^{TM}) \mathrm{ch}_2(E,\nabla ^{E}).\ee

Let $\Lambda_\CC (T^*M)$ be the complexified exterior algebra bundle of $TM$. Let $\langle \ , \
\rangle_{\Lambda_\CC (T^*M)}$ be the Hermitian metric on $\Lambda_\CC(T^*M)$
induced by $g^{TM}$. Let $dv$  be the Riemannian volume form associated to $g^{TM}$. Then $\Gamma(M, \Lambda_\CC (T^*M))$ has a Hermitian metric such that
for $\alpha, \alpha'\in \Gamma(M, \Lambda_\CC (T^*M))$,
$$\langle \alpha, \alpha'\rangle=\int_{M}\langle \alpha,
 \alpha'\rangle_{\Lambda_\CC (T^*M)}\,dv.$$

For $X\in TM$, let $c(X)$ be the Clifford action on $\Lambda_\CC (T^*M)$ defined by $c(X)=X^*-i_X$, where $X^*\in T^*M$ corresponds to $X$ via
$g^{TM}$. Let $\{e_1,e_2, \cdots, e_{2n}\}$ be an oriented orthogonal basis of $TM$. Set
$$\Omega=(\sqrt{-1})^{n}c(e_1)\cdots c(e_{2n}).$$ Then one can show that $\Omega$ is independent of the choice of the orthonormal basis and $\Omega_E=\Omega\otimes 1$ is a self-adjoint operator on $\Lambda_\CC (T^*M)\otimes E$ such that $\Omega_E^2=\mathrm{Id}|_{\Lambda_\CC(T^*M)\otimes E}$.

Let $d$ be the exterior differentiation operator and $d^*$ be the formal adjoint of $d$ with respect to the Hermitian metric.  The operator
$$D_{Sig}:=d+d^*=\sum_{i=1}^{2n} c(e_i)\nabla_{e_i}^{\Lambda_\CC (T^*M)}: \Gamma(M, \Lambda_\CC (T^*M))\to \Gamma(M, \Lambda_\CC (T^*M))$$ is the signature operator and the more general twisted signature operator is defined as (cf. \cite{G})
$$ D_{Sig}\otimes E:=\sum_{i=1}^{2n} c(e_i)\nabla_{e_i}^{\Lambda_\CC (T^*M)\otimes E}: \Gamma(M, \Lambda_\CC (T^*M)\otimes E)\to \Gamma(M, \Lambda_\CC (T^*M)\otimes E).$$

The operators $ D_{Sig}\otimes E$ and $\Omega_E$ are anti-commutative. If we decompose
$\Lambda_\CC (T^*M)\otimes E=\Lambda^+_\CC (T^*M)\otimes E\oplus \Lambda^-_\CC (T^*M)\otimes E$ into $\pm 1$ eigenspaces of $\Omega_E$, then $D_{Sig}\otimes E$ decomposes to define
\be (D_{Sig}\otimes E)^{\pm}: \Gamma(M, \Lambda^{\pm}_\CC (T^*M)\otimes E)\to \Gamma(M, \Lambda^{\mp}_\CC (T^*M)\otimes E).\ee

The {\it twisted signature} of $M$ is defined as the index of the operator $(D_{Sig}\otimes E)^{+}$ denoted by $\sig(M, E)$, 
\be \sig(M, E)=\ind((D_{Sig}\otimes E)^{+}).\ee
By the Atiyah-Singer index theorem, 
$$\sig(M, E)=\int_M  \widehat{L}(TM,\nabla ^{TM})\mathrm{ch}(E, \nabla^E).$$ Note that in the book \cite{Law} (Theorem 13.9), the following formula is given
$$\sig(M, E)=\int_M L(TM,\nabla ^{TM}) \mathrm{ch}_2(E,\nabla ^{E}).$$ 

There is an important twisted $\widehat{A}$-genus, namely the Witten genus \cite{W} by coupling $\widehat{A}(M)$ with {\it the Witten bundle} \cite{W}
\[
\Theta (T_{\mathbb{C}}M)=\overset{\infty }{\underset{n=1}{\otimes }}%
S_{q^{2n}}(\widetilde{T_{\mathbb{C}}M}),\ \ {\rm with}\ \
\widetilde{T_{\mathbb{C}}M}=TM\otimes \mathbb{C}-{\mathbb C}^{4m}.
\]
{\it The Witten genus} then can defined as
\[
W(M)=\left\langle \widehat{A}(TM)\mathrm{ch}\left( \Theta \left( T_{\mathbb{C%
}}M\right) \right) ,[M]\right\rangle.
\]

\end{document}